\documentclass[oneside]{amsart}
\UseRawInputEncoding
\usepackage[utf8]{inputenc}   
\usepackage[T1]{fontenc}      
\usepackage{geometry}                
\usepackage{graphicx}                
\usepackage{amssymb, amsmath}        
\usepackage{enumitem}                
\usepackage{hyperref}                
\usepackage{epstopdf}                

\newtheorem{theorem}{Theorem}[section]
\newtheorem{lemma}[theorem]{Lemma}

\newtheorem{corollary}[theorem]{Corollary}

\newtheorem{definition}[theorem]{Definition}

\newtheorem*{theorem*}{Theorem}   

\newcommand{\proj}{\text{proj}}

\title{A Fraction Dimensional Kakeya Compression}
\author{Jiahan Du}
\address{University of California, Los Angeles}
\email{jiahandu@math.ucla.edu}
\date{\today}

\begin{document}

\begin{abstract}
Recently, it was shown in \cite{guo2026curvedkakeyaproblemsprojective} that under a very general condition,  a generic family of curved Kakeya sets in $\mathbb{R}^3$ not trivially contained in a hypersurface has Hausdorff dimension strictly greater than 2. In this paper we leverage large language models to give an explicit construction of a family of curves failing the worst compression phenomenon in \cite{guo2026curvedkakeyaproblemsprojective}, but whose union has Hausdorff dimension and upper box dimension $5/2$. The example shows that there are Kakeya compression phenomena besides the trivial example of concentrating near a hypersurface. This answers a question raised in section 1.9 of \cite{guo2026curvedkakeyaproblemsprojective}. The models used are Chatgpt 5.6 Sol and Chatgpt 6 Astra.  
\end{abstract}

\maketitle

\section{Introduction}

Let \(\phi(x,t;y)\) be real analytic near the origin, where
\[
x,y\in\mathbb R^2,\qquad t\in\mathbb R,\qquad z=(x,t)\in\mathbb R^3.
\]Define
\[
G(z;y)
=
\partial_{y_1}\nabla_z\phi(z;y)
\times
\partial_{y_2}\nabla_z\phi(z;y).
\]Assume Hörmander’s conditions hold throughout the neighbourhood:
\[
\det\!\left[
\partial_{x_i}\partial_{y_j}\phi(x,t;y)
\right]_{i,j=1}^{2}\ne0,
\tag{H1}
\]and
\[
\det\!\left[
G(z;y)\cdot
\partial_{y_i}\partial_{y_j}\nabla_z\phi(z;y)
\right]_{i,j=1}^{2}\ne0.
\tag{H2}
\]Define the associated curve map locally by
\[
\nabla_y\phi\bigl(X(v,t;y),t;y\bigr)=v,
\qquad
\Phi(v,t;y)=\bigl(X(v,t;y),t\bigr),
\qquad v\in\mathbb R^2.
\]
Assume this family has the normal form:
\[
\Phi(w,t;\xi)=(X(w,t;\xi),t),
\]\[
X(w,t;\xi)
=
w-t\bigl(\xi+R(w,\xi)\bigr)+t^2S(w,t,\xi),
\]where \(R,S\) are analytic functions and
\[
|R(w,\xi)|\le C|w||\xi|,
\qquad
|S(w,t,\xi)|\le C|\xi|.
\]Fix sufficiently small \(\varepsilon>0\), and set
\[
B=\{a\in\mathbb R^2:|a|<\varepsilon\},
\qquad I=[-\varepsilon,\varepsilon].
\]Given \(w_0,\eta,\zeta\in B\), with \(\eta\ne\zeta\), let \(D\) comprise the pairs \((s,t)\in I^2\), \(s\ne t\), for which
\[
\begin{cases}
X(w,s;\xi)=X(w_0,s;\eta),\\
X(w,t;\xi)=X(w_0,t;\zeta)
\end{cases}
\]has a unique solution \((w,\xi)\in B\times B\). Write this solution as \((w(s,t),\xi(s,t))\).

We start with following definition taken from \cite{guo2026curvedkakeyaproblemsprojective}. 
\begin{definition}[Worst Kakeya Compression, Definition 1.23, \cite{guo2026curvedkakeyaproblemsprojective}]

The phase \(\phi\) satisfies worst compression if \(w_0,\eta,\zeta\) can be chosen such that:
\begin{enumerate}
\item The union
   \[
   \bigcup_{(s,t)\in D}
   \left\{
   \bigl(X(w(s,t),u;\xi(s,t)),u\bigr):u\in I
   \right\}
   \]forms a two-dimensional surface in \(\mathbb R^3\).
\item At some \((s_0,t_0)\in D\), the solution parameters are smooth in a neighborhood and
   \[
   \det\!\left[
   \partial_s\xi(s_0,t_0)\ \ 
   \partial_t\xi(s_0,t_0)
   \right]\ne0.
   \]
\end{enumerate}
\end{definition}

If we have worst Kakeya compression, then the curved Kakeya set is contained in a hypersurface, and has Hausdorff dimension 2. In this paper, we shall give an explicit construction that demonstrates the existence of fraction dimensional Kakeya compression, Theorem \ref{main}. \\

\textbf{Acknowledgment.} I thank Arian Nadjimzadah for informing me of concurrent independent work of Shaoming Guo and Mingfeng Chen, who have obtained a related, and possibly more general, construction, also with the assistance of AI.

\section{A Family of Curves}
Write $\omega=(\omega_1,\omega_2)$ for position parameters  and $\xi=(p,q)$ for direction parameters. Consider the map
\begin{equation}\label{curves}
    \Phi(\omega,t;p,q)=\bigg(\omega_1-tp-t^2\bigg(q+\frac{p^2}{2}\bigg)-t^3p, \omega_2-tq-t^2p, t \bigg).
\end{equation}

A direct computation shows that $\Phi$ is associated with the following phase function $\phi$,

\begin{equation*}
    \phi(x_1,x_2,t;p,q)=x_1p+x_2q+\frac{t}{2}(p^2+q^2)+t^2\bigg(pq+\frac{p^3}{6}\bigg)+\frac{t^3}{2}p^2. 
\end{equation*}
Another direct computation shows that $\phi$ satisfies the H\"ormander condition\footnote{Our choice of the phase function $\phi$ is motivated by equation 1.77 of \cite{guo2026curvedkakeyaproblemsprojective}. }.  

\begin{lemma}
There does not exist a nonempty open set $U\subset\mathbb{R}^2$ and a $C^2$ map $W:U\rightarrow\mathbb{R}^2$ such that the map 
\begin{equation}
    \Omega: (p,q,t)\longmapsto (W_1(p,q)-tp-t^2(q+p^2/2)-t^3p, W_2(p,q)-tq-t^2p, t)
\end{equation}
 has image contained in a two dimensional surface for all $t$ in a nonempty open interval. Furthermore, this implies that $\phi$ fails the worst Kakeya compression definition. 
\end{lemma}
\begin{proof}
    We assume otherwise. Surface containment forces 
    \begin{equation*}
        \det(DW-H_p(t))=0, \;\;\;\;\;\;\;\; H_p(t)=\begin{bmatrix}
            t+pt^2+t^3 & t^2\\
            t^2 & t
        \end{bmatrix}. 
    \end{equation*}
    Indeed, if $\det(DW-H_p(t))\not=0$ for some $t$, then $D\Omega$ would have rank 3 at that t, and its image would contain a nonempty open set by the inverse function theorem. 

    Write
    \begin{equation*}
        DW=\begin{bmatrix}
            A & B\\
            C & D
        \end{bmatrix}.
    \end{equation*}
    Expanding $\det(DW-H_p(t))=0$ as a polynomial of $t$ gives 
    \begin{equation*}
        AD-BC-(A+D)t+(1+B+C-pD)t^2+(p-D)t^3=0. 
    \end{equation*}
    Therefore 
    \begin{equation*}
        A=-p,\;\;\;\;\;D=p,\;\;\;\;\;B+C=p^2-1, \;\;\;\;\; BC=-p^2. 
    \end{equation*}
    It follows that 
    \begin{equation*}
        DW=\begin{bmatrix}
            -p & -1\\
            p^2 & p
        \end{bmatrix} \;\;\;\;
        \text{ or }\;\;\;\;
        DW= \begin{bmatrix}
            -p & p^2 \\
            -1 & p
        \end{bmatrix}.
    \end{equation*}
    In either case, we have $\partial_p\partial_q W_2=\partial_p D=1$, $\partial_q\partial_p W_2=\partial_q C=0$, contradicting equality of mixed partials. This proves that such $U$ and $W$ do not exist, which implies that $\phi$ fails the worst Kakeya compression condition. 
\end{proof}

Given $n\in\mathbb{Z}^+$, $\epsilon>0$, we denote by $B^n_\epsilon$ the open ball in $\mathbb{R}^n$ centered at the origin of radius $\epsilon$. 

\begin{theorem}\label{main}
    The family of curves in equation \ref{curves} satisfies H\"ormander's condition and fails the worst compression condition. Take $\varepsilon=1/100$ to be a small but absolute constant. There exists a compact parameter set $E\subset\mathbb{R}^4$ and a compact curved Kakeya set $K\subset\mathbb{R}^3$
    \begin{enumerate}
        \item For every $(p,q)\in\overline{B_\varepsilon^2}$,
there exists $\omega\in B_{2\varepsilon}^2$
such that $(\omega,p,q)\in E$;
        \item $\dim_P E=2$;
        \item $K$ is the union of $\Phi(\omega,t; p,q)$ for every $(\omega,p,q)\in E$ and every $|t|\leq\varepsilon$;
        \item $\dim_H K=\overline{\dim_B} K=5/2$.
    \end{enumerate}
\end{theorem}

To prove Theorem \ref{main}, we first construct a compact set $Q\subset\mathbb{R}^3$, based on the Heisenberg square construction in \cite{article}. For every pair of sequences $e=(e_j)_{j\geq 1}, f=(f_j)_{j\geq 1}$ consisting of $0$ and $1$ only, put
\begin{equation}
    x=\sum_{j\geq 1}2^{-j}e_j, \;\;\;\;\;y=\sum_{j\geq 1}2^{-j}f_j,\;\;\;\;\;z=\sum_{i<j}2^{-i-j}e_if_j,
\end{equation}
and let $Q$ be the set of all such triples. $Q$ is closed and bounded hence compact. Moreover 
\begin{equation}
    \proj_{x,y}Q=[0,1]^2,\;\;\;\;\;0\leq z\leq\sum_{i<j}2^{-i-j}=\frac{1}{3}. 
\end{equation}

Since $\proj_{x,y}(Q)=[0,1]^2$, we have $\dim_H Q\geq 2$. Given $n\in\mathbb{Z}^+$, we fix 
\[
e_1,\ldots,e_n,\qquad f_1,\ldots,f_n,
\]and let the remaining digits vary. Writing \(r=2^{-n}\) and
\[
x_0=\sum_{j\le n}2^{-j}e_j,\quad
y_0=\sum_{j\le n}2^{-j}f_j,\quad
z_0=\sum_{i<j\le n}2^{-i-j}e_if_j,
\]the resulting piece of \(Q\) is exactly
\begin{equation}\label{piece}
\left\{
\left(x_0+rx',\ y_0+ry',\
z_0+rx_0y'+r^2z'\right):
(x',y',z')\in Q
\right\}.
\end{equation}
From here we can deduce that each such piece of $Q$ has diameter $O(r)$. Since there are $O(r^{-2})$ number of such pieces, we see that $\overline{\dim_B} Q\leq 2$. This implies 
\begin{equation*}
    \dim_H Q=\dim_P Q=\overline{\dim_B} Q=2. 
\end{equation*}

Recall $\varepsilon=1/100$, we take $s=4\varepsilon$. Define the rescaled version of $Q$ by: 
\begin{equation}\label{qs}
    Q_s=\big\{(p,a,b)=\big(s(x-\frac{1}{2}), s(y-\frac{1}{2}), s^2(z-\frac{1}{2}y)\big):(x,y,z)\in Q\big\}. 
\end{equation}
Then $\proj_{p,a}Q_s=[-s/2,s/2]^2$, and $|b|\leq s^2/2$ if $(p,a,b)\in Q_s$. Define
\[
E=\left\{
\left(-a,b,p,a-\frac{p^2}{2}\right):
(p,a,b)\in Q_s
\right\}.
\] Clearly 
\begin{equation*}
    \dim_H E=\dim_P E=\overline{\dim_B} E=2. 
\end{equation*}
This is (2) of Theorem \ref{main}. (1) of Theorem \ref{main} follows from a direct computation and the fact that $\proj_{p,a}Q_s=[-s/2,s/2]^2$. Define \(K\) to be the union of curve segments for \(|t|\le\varepsilon\). 

\section{$\dim_H K=5/2$}

For a point 
\begin{equation*}
\left(-a,b,p,a-\frac{p^2}{2}\right)
\end{equation*}
in $E$, the first two coordinates of the curves in equation \ref{curves} simplify to 
\begin{equation*}
    \Phi_1=-(1+t^2)(a+tp),\;\;\;\;\;\;\;\;\Phi_2=b-t(a-p^2/2)-t^2p. 
\end{equation*}

Define the following diffeomorphism on $\mathbb{R}^3$, 
\begin{equation*}
    \Psi(x_1,x_2,t)=\bigg(-\frac{x_1}{1+t^2}, x_2-\frac{tx_1}{1+t^2},t\bigg)
\end{equation*}
sends the curve $t\mapsto(\Phi_1,\Phi_2,t)$ to the following straight line 
\begin{equation*}
    l_{p,a,b}(t)=\bigg(a+tp, b+\frac{t}{2}p^2,t\bigg).
\end{equation*}
Therefore, we only need to show that if
\begin{equation*}
K'=\{l_{p,a,b}(t)|\,(p,a,b)\in Q_s, |t|\leq\varepsilon\},
\end{equation*}
then $\overline{\dim_B}K'\leq5/2$. To see why, for a point $(x,y,z)$ belonging to a piece such as equation \ref{piece}, we have 
\begin{gather}
    |x-x_0|+|y-y_0|+|z-z_0|\lesssim r \\
    |z-z_0-x_0(y-y_0)|\lesssim r^2. 
\end{gather}
Recall $Q_s$ is obtained from $Q$ under an affine map, let $(p_0,a_0,b_0)$ and $(p,a,b)$ be the images of $(x_0,y_0,z_0)$ and $(x,y,z)$ respectively. Let $R=s2^{-n}$, then a direct computation shows that 
\begin{gather}
    |p-p_0|+|a-a_0|+|b-b_0|\lesssim R\\
    |b-b_0-p_0(a-a_0)|\lesssim R^2. 
\end{gather}
Also notice that there are $4^n=s^2R^{-2}$ such pieces. Define 
\begin{equation*}
    U=(a+tp)-(a_0+tp_0),\;\;\;\;\;\;V=(b+tp^2/2)-(b_0+tp_0^2/2). 
\end{equation*}
Another direct computation shows that 
\begin{equation*}
|U|\lesssim R, \;\;\;\;\;\;\;|V-p_0U|=\bigg|(b-b_0-p_0(a-a_0))+\frac{t}{2}(p-p_0)^2\bigg|\lesssim R^2. 
\end{equation*}
Consequently, the line $l_{p,a,b}$ is covered in a rectangular box of dimensions $O(1)\times O(R)\times O(R^2)$. Since there are $s^2R^{-2}\lesssim R^{-2}$ such pieces as equation \ref{piece}, we see that the union of all segments $K'$ is covered in $R^{-2}$ number of rectangular boxes of dimension $O(1)\times O(R)\times O(R^2)$. Each box of dimension $O(1)\times O(R)\times O(R^2)$ is covered by $O(R^{-3})$ number of balls of diameter $R^2$. Therefore $K'$ is covered by $O(R^{-5})$ number of balls of radius $R^2$. This allows us to conclude that $\overline{\dim_B}K'\leq 5/2$. 

The matching Hausdorff dimension lower bound for $K$ comes from results in \cite{nadjimzadah2026newcurvedkakeyaestimates}. Denote the unit ball in $\mathbb{R}^n$ by $B_n$. 

\begin{definition}[Definition 2.1, \cite{nadjimzadah2026newcurvedkakeyaestimates}]
    Let $X(\boldsymbol{p},t):B_3\times B_1\to\mathbb{R}^2$ be a smooth map which we call the defining function. The point $\boldsymbol{p}\in B_3$ determines the following curve segment 
    \begin{equation}
        l_{\boldsymbol{p}}(t)=(X(\boldsymbol{p},t),t)
    \end{equation}
    and $t\in B_1$. Write $\mathcal{C}(X)$ for the 3-parameter family of these curves. 
\end{definition}

\begin{definition}[Definition 2.2, \cite{nadjimzadah2026newcurvedkakeyaestimates}]
Write $\boldsymbol{p}=(p,\vec{a})\in\mathbb{R}\times\mathbb{R}^2$, where $\vec{a}=(a,b)$, and define the curves $\omega_{\boldsymbol{p},t}(p')=(\partial_t X)(p',\vec{a}_{\boldsymbol{p},t}(p'),t)$, where $\vec{a}_{\boldsymbol{p},t}$ solves
\begin{equation}
    X(p',\vec{a}_{\boldsymbol{p},t}(p'),t)=X(\boldsymbol{p},t). 
\end{equation}
\end{definition}

\begin{definition}[Condition 2, \cite{nadjimzadah2026newcurvedkakeyaestimates}]
Given $c>0$, the map $X$ is $c$-coney if for each point $(\boldsymbol{p},t)$ we have  
\begin{equation}
    \left|
\det\left(
\left.\frac{d}{dp'}\omega_{\boldsymbol p,t}(p')\right|_{p'=p},
\left.\frac{d^2}{d(p')^2}\omega_{\boldsymbol p,t}(p')\right|_{p'=p}
\right)
\right|\ge c.
\end{equation}
\end{definition}

\begin{corollary}[Corollary 2.10, \cite{nadjimzadah2026newcurvedkakeyaestimates}]\label{Arian}
Let $F_t:\mathbb{R}^3\to\mathbb{R}^2$ be a family of maps, and define $X(\boldsymbol{p},t)=F_t(\boldsymbol{p})$. Let $A$ be a Borel subset of $\mathbb{R}^3$ with $\dim_H A=2$. If $X$ satisfies the "basic conditions" defined in \cite[Condition 4]{nadjimzadah2026newcurvedkakeyaestimates} and is $c$-coney for some positive constant $c>0$, then $\dim_H F_t(A)\geq 1.5$ for a.e. $t\in[0,1]$. 
\end{corollary}

Define 
\begin{equation*}
F_t(p,a,b)=\bigg(a+tp, b+\frac{t}{2}p^2 \bigg).
\end{equation*}

Since $\dim_H Q_s=2$, and a direct calculation shows that $\omega_{\boldsymbol{p},t}(p')=(p',\frac{p'^2}{2})$. Therefore Corollary \ref{Arian} applies, and we can conclude that 
\begin{equation*}
    \dim_H F_t(Q_s)\geq\frac{3}{2} \;\;\;\text{for almost every }t\in[0,\varepsilon]. 
\end{equation*}

Since above holds for almost every $t$, we see that $\dim_H K'\geq\frac{3}{2}+1$. This concludes the lower bound proof and the proof of Theorem \ref{main}. 

\bibliographystyle{ieeetr} 
\bibliography{references}
\end{document}